\documentclass[11pt]{article}
\usepackage{amsfonts}
\usepackage{amsmath}
\usepackage{geometry}

\newtheorem{theorem}{Theorem}

\newtheorem{remark}[theorem]{Remark}

\newenvironment{proof}[1][Proof]{\noindent\textbf{#1.} }{\ \rule{0.5em}{0.5em}}
\input{tcilatex}
\begin{document}

\title{A criterion for the triviality of non-steady gradient Ricci solitons}
\author{Mohammed Guediri}
\maketitle

\begin{abstract}
The main purpose of this paper is to show that a normalized non-steady
gradient Ricci soliton $\left( M,g,f,\lambda \right) $ of dimension $n$ is
trivial if and only if its scalar curvature $S$ satisfies the equality $%
S=\lambda \left( f+\frac{n}{2}\right) .$
\end{abstract}

\smallskip \renewcommand{\thefootnote}{} 

\section{Introduction}

A one-parameter family of Riemannian metrics $g(t)$ on a smooth manifold $M$
is a solution of the Ricci flow equation if 
\begin{equation}
\frac{\partial }{\partial t}g\left( t\right) =-2Ric_{g\left( t\right) }
\label{1}
\end{equation}%
where $Ric_{g\left( t\right) }$ is the Ricci curvature of $g\left( t\right)
. $

A solution $g(t)$ of the Ricci flow on $M$ is said to be a self-similar
solution if there exist a positive function $\sigma \left( t\right) $ and a
one-parameter family $\psi \left( t\right) $ of diffeomorphisms of $M$ such
that the one-parameter family of metrics $g\left( t\right) =\sigma \left(
t\right) \psi \left( t\right) ^{\ast }g\left( 0\right) $ satisfies equation (%
\ref{1}). Here $\psi \left( t\right) ^{\ast }$ stands for the pullback along
the diffeomorphism $\psi \left( t\right) .$

Given a Riemannian manifold $\left( M,g\right) ,$ it can be shown that for
any self-similar solution $g\left( t\right) $ of the Ricci fow, with the
initial condition $g\left( 0\right) =g,$ there exists a vector field $X$ on $%
M$ and a constant $\lambda $ satisfying the equation

\begin{equation}
Ric+\frac{1}{2}L_{X}g=\lambda g,  \label{2}
\end{equation}%
where $Ric$ is the Ricci curvature of $g$ and $L_{X}g$ is the Lie derivative
of $g$ in the direction of $X.$ Conversely, given a vector field $X$ on $M$
satisfying (\ref{2}) for some constant $\lambda ,$ it can be shown that $X$
generates a one-parameter family $\psi \left( t\right) $ of diffeomorphisms
so that the one-parameter family of metrics 
\begin{equation*}
g\left( t\right) =\left( 1-2\lambda t\right) \psi \left( t\right) ^{\ast }g
\end{equation*}%
is a self-similar solution of the Ricci flow given by equation (\ref{1}).
For more details about this correspondence, see for instance \cite{chow}.

Given a Riemannian manifold $\left( M,g\right) $ and a vector field $X$ on $%
M $ satisfying equation (\ref{2}) for some constant $\lambda ,$ we say that
the quadruple $\left( M,g,X,\lambda \right) $ is a Ricci soliton. It is
called shrinking if $\lambda >0,$ steady if $\lambda =0,$ and expanding if $%
\lambda <0.$

\section{Preliminaries}

If $X$ is the gradient $\nabla f$ of some function $f$ on $M$ with respect
to the metric $g,$ then $\left( M,g,f,\lambda \right) $ is called a gradient
Ricci soliton. Since the divergence of a vector field $X$ is defined as 
\begin{equation*}
\func{div}X=\frac{1}{2}trace\left( L_{X}g\right) ,
\end{equation*}%
and since the hessian of a function $f$ with respect to $g$ is defined as $%
Hess~f=\func{div}X,$ it follows that for a gradient Ricci soliton $\left(
M,g,f,\lambda \right) ,$ equation (\ref{1}) states the form

\begin{equation}
Ric+Hess~f=\lambda g.  \label{3}
\end{equation}

If $f$ is constant, we say that the soliton is trivial. In this case, $%
\left( M,g\right) $ becomes an Einstein manifold.

\bigskip

On the other hand, if $\left( M,g,X,\lambda \right) $ is a Ricci soliton,
then by tracing (\ref{2}), we get%
\begin{equation}
S+\func{div}X=n\lambda ,  \label{4}
\end{equation}%
where $S$ denotes the scalar curvature of $g.$ In the particular case of a
gradient Ricci soliton $\left( M,g,f,\lambda \right) ,$ we get

\begin{equation}
S+\triangle f=n\lambda  \label{5}
\end{equation}

We also have the following well known identity that follows from the second
Bianchi identity 
\begin{equation}
\nabla S=2\func{div}\left( Ric\right)  \label{6}
\end{equation}

For a gradient Ricci soliton $\left( M,g,f,\lambda \right) ,$ by using (\ref%
{2}) and (\ref{6}), we easily get%
\begin{equation}
S+\left\vert \nabla f\right\vert ^{2}=2\lambda f+c,  \label{7}
\end{equation}%
where $c$ is a constant.

If $\left( M,g,f,\lambda \right) $ is a non-steady gradient Ricci soliton,
then one can replace $f$ by $f-\frac{c}{2\lambda }$ to obtain%
\begin{equation}
S+\left\vert \nabla f\right\vert ^{2}=2\lambda f  \label{7'}
\end{equation}

In this case, we say that $\left( M,g,f,\lambda \right) $ is a normalized
gradient Ricci soliton.

\bigskip If $\lambda _{1},\ldots ,\lambda _{n}$ are the eigenvalues of the
Ricci $\left( 1,1\right) $-tensor, we define the squared norm of $Ric$ as 
\begin{equation*}
\left\vert Ric\right\vert ^{2}=\sum_{i=1}^{n}\lambda _{i}^{2}
\end{equation*}

In a similar way, we define the squared norm $\left\vert Hess~f\right\vert
^{2}$ of the hessian of $f$ to be the squared norm of the operator $\nabla
\nabla f.$

Using the Cauchy--Schwarz inequality, we get 
\begin{equation}
\left\vert Ric\right\vert ^{2}\geq \frac{S^{2}}{n}  \label{8}
\end{equation}

We finally turn our attention to the following formulas, quoted from \cite%
{petersen-wylie}, which will be useful for us in the next sections. 
\begin{equation}
\triangle S-g\left( \nabla S,\nabla f\right) +2\left\vert Ric\right\vert
^{2}=2\lambda S  \label{9}
\end{equation}

\begin{equation}
\frac{1}{2}\triangle \left( \left\vert \nabla f\right\vert ^{2}\right)
=\left\vert Hess~f\right\vert ^{2}-Ric\left( \nabla f,\nabla f\right)
\label{10}
\end{equation}

In fact, it should be noticed that formula (\ref{10}) is nothing but the
well known Bochner's formula adapted to gradient Ricci solitons.

\section{\protect\bigskip Gradient Ricci solitons satisfying the equality $S=%
\protect\lambda f+c$}

\begin{theorem}
Let $\left( M,g,f,\lambda \right) $ be a normalized non-steady gradient
Ricci soliton of dimension $n$, which satisfies an equality of the form $%
S=\lambda f+c,$ for some constant $c.$ Then, we have%
\begin{eqnarray}
\left\vert Ric\right\vert ^{2} &=&\lambda ^{2}\left( 2f-\frac{n}{2}\right)
+\lambda c,  \label{11} \\
\left\vert Hess~f\right\vert ^{2} &=&\lambda \left( \frac{n}{2}\lambda
-c\right)  \label{12}
\end{eqnarray}
\end{theorem}

\begin{proof}
Note first that, from the equality $S=\lambda f+c,$ we have $\nabla
S=\lambda \nabla f$ and $\triangle S=\lambda \triangle f$ . Taking into
account these quantities, we see that (\ref{9}) becomes

\begin{equation}
\lambda \triangle f-\lambda \left\vert \nabla f\right\vert ^{2}+2\left\vert
Ric\right\vert ^{2}=2\lambda \left( \lambda f+c\right)  \label{13}
\end{equation}

On the other hand, by substituting the value $S=\lambda f+c$ into (\ref{5})
and (\ref{7'}), then we get

\begin{equation}
\triangle f=n\lambda -\lambda f-c,  \label{14}
\end{equation}%
and%
\begin{equation}
\left\vert \nabla f\right\vert ^{2}=\lambda f-c,  \label{15}
\end{equation}%
respectively.

By replacing (\ref{14}) and (\ref{15}) in (\ref{13}), we easily obtain (\ref%
{11}).

Similarly, we derive (\ref{12}) just by replacing (\ref{14}) and (\ref{15})
in the adapted Bochner's formula (\ref{10}) and taking into account the
following formula which is nothing but formula (\ref{6}). 
\begin{equation*}
Ric\left( \nabla f,\nabla f\right) =\frac{1}{2}g\left( \nabla f,\nabla
S\right) .
\end{equation*}
\end{proof}

\begin{remark}
It is remarkable that, adding equations\bigskip \bigskip\ (\ref{11}) and (%
\ref{12}) gives the equation%
\begin{equation}
\left\vert Ric\right\vert ^{2}+\left\vert Hess~f\right\vert ^{2}=2\lambda
^{2}f  \label{16}
\end{equation}
\end{remark}

\section{Main result\protect\bigskip}

We are now in position to prove our main result. Its proof will rely in a
very strong way on the formulas~(\ref{11}), (\ref{12}), and (\ref{16}).

\begin{theorem}
\label{main thm}Let $\left( M,g,f,\lambda \right) $ be a normalized
non-steady gradient Ricci soliton of dimension $n,$ and let $S$ denote the
scalar curvature of $M.$ Then, $\left( M,g,f,\lambda \right) $ is trivial if
and only if $S=\lambda \left( f+\frac{n}{2}\right) .$
\end{theorem}

\begin{proof}
Assume first that $\left( M,g,f,\lambda \right) $ is trivial. In this case,
we deduce from (\ref{5}) and (\ref{7'}) that $S=n\lambda $ and $f=\frac{n}{2}%
,$ respectively. Therefore, 
\begin{equation*}
S=\lambda \left( f+\frac{n}{2}\right) .
\end{equation*}

Conversely, assume that 
\begin{equation*}
S=\lambda \left( f+\frac{n}{2}\right) .
\end{equation*}
Then, with $c=\frac{n}{2}\lambda ,$ we deduce from (\ref{12}) that $%
\left\vert Hess~f\right\vert =0.$ It follows from (\ref{16}) that 
\begin{equation*}
\left\vert Ric\right\vert ^{2}=2\lambda ^{2}f.
\end{equation*}

Now, by using (\ref{8}), we have

\begin{equation*}
2\lambda ^{2}f=\left\vert Ric\right\vert ^{2}\geq \frac{S^{2}}{n}=\frac{%
\lambda ^{2}\left( f+\frac{n}{2}\right) ^{2}}{n},
\end{equation*}%
from which we get that 
\begin{equation*}
2nf\geq f^{2}+\frac{n^{2}}{4}+nf,
\end{equation*}%
or equivalently%
\begin{equation*}
0\geq \left( f-\frac{n}{2}\right) ^{2}.
\end{equation*}

Therefore $f=\frac{n}{2},$ namely $\left( M,g,f,\lambda \right) $ is trivial.
\end{proof}

\bigskip

Theorem \ref{poisson thm} below will provide a good example of a non-trivial 
$\left( M,g,f,\lambda \right) .$ But, let us first turn our attention to a
related result which was first proved in \cite{chen-deshmukh}. We give here
a very short and simple proof of it.

\begin{theorem}
Let $\left( M,g,f,\lambda \right) $ be a compact shrinking gradient Ricci
soliton of dimension $n.$ If the scalar curvature $S$ is solution of the
poisson equation $\triangle S=\sigma ,$ where $\sigma =\lambda \left(
n\lambda -S\right) ,$ then either $\left( M,g,f,\lambda \right) $ is trivial
or $\lambda \geq \lambda _{1},$ where $\lambda _{1}$ is the first eigenvalue
of the Laplacian $\triangle .$
\end{theorem}

\begin{proof}
Since 
\begin{eqnarray*}
\triangle S &=&\sigma \\
&=&\lambda \left( n\lambda -S\right) ,
\end{eqnarray*}
we get 
\begin{equation}
-\triangle \left( n\lambda -S\right) =\lambda \left( n\lambda -S\right)
\label{17}
\end{equation}

If $S$ is constant, it follows from (\ref{17}) that $S=n\lambda .$ By
substituting this into (\ref{5}), we get 
\begin{equation*}
\triangle f=n\lambda -S=0.
\end{equation*}%
Since $M$ is compact, we deduce that $f$ is constant, that is, $\left(
M,g,f,\lambda \right) $ is trivial.

If $S$ is not constant, then we see from (\ref{17}) that $\lambda $ is an
eigenvalue of the Laplacian $\triangle .$ Consequently, we have $\lambda
\geq \lambda _{1}.$
\end{proof}

\bigskip

\begin{theorem}
\label{poisson thm}Let $\left( M,g,f,\lambda \right) $ be a compact
normalized shrinking non-trivial gradient Ricci soliton of dimension $n.$ If
the scalar curvature $S$ is solution of the poisson equation $\triangle
S=\sigma ,$ where $\sigma =\lambda \left( n\lambda -S\right) ,$ then $%
\lambda $ is an eigenvalue of the Laplacian $\triangle $ and $S=\lambda f+c,$
for some constant $c\neq \frac{n}{2}\lambda .$
\end{theorem}

\begin{proof}
Since $\left( M,g,f,\lambda \right) $ is not trivial, it follows that $S$ is
not constant. Therefore, as in the proof of the above theorem, we deduce
that $\lambda $ is an eigenvalue of the Laplacian $\triangle .$

On the other hand, by (\ref{5}), we have 
\begin{eqnarray*}
\triangle S &=&\lambda \left( n\lambda -S\right) \\
&=&\lambda \triangle f
\end{eqnarray*}

In other words, we have 
\begin{equation*}
\triangle \left( S-\lambda f\right) =0.
\end{equation*}

Since $M$ is compact, we deduce that $S-\lambda f=c$ for some constant $c.$
By Theorem \ref{main thm}, we shoud have $c\neq \frac{n}{2}\lambda .$
\end{proof}

\bigskip

Mohammed Guediri

Department of Mathematics, College of Science,

King Saud University.

PO Box 2455, Riyadh 11451, Saudi Arabia

\bigskip E-mail addresses: mguediri@ksu.edu.sa.

\medskip


\bigskip

\begin{thebibliography}{9}
\bibitem{chen-deshmukh} B.Y. Chen and S. Deshmukh, \textsl{Geometry of
compact shrinking Ricci solitons,} Balkan J. Geom. Applications, \textbf{19 }%
(2014), 13-21.

\bibitem{chow} B. Chow ... [et al.], \textsl{The Ricci Flow: Techniques and
Applications: Part I: Geometric Aspects, }Mathematical Surveys and
Monographs, Volume: 135; 2007.

\bibitem{petersen-wylie} P. Petersen and W. Wylie, \textsl{Rgidity of
gradient Ricci solitons, }Pacific J. Math. \textbf{241 }(2009), 329-345.
\end{thebibliography}
\end{document}